\documentclass[12pt]{article}
\usepackage{graphicx}
\usepackage{amsmath,amsthm,amssymb,amsfonts,euscript,enumerate}
\usepackage{amsthm}
\usepackage{color,wrapfig}
\usepackage{pxfonts}
\usepackage{verbatim}
\usepackage{float}
\floatstyle{plaintop}
\restylefloat{table}
\usepackage[tableposition=top]{caption}

\usepackage{fancyhdr}
\newtheorem{theorem}{Theorem}[section]
\newtheorem{conjecture}{Conjecture}[section]
\newtheorem{corollary}[theorem]{Corollary}
\newtheorem{lemma}[theorem]{Lemma}

\newtheorem{definition}{Definition}[section]

\DeclareMathOperator{\ord}{ord}
\DeclareMathOperator{\lcm}{lcm}
\DeclareMathOperator{\sgn}{sgn}

\begin{document}
\title{The fundamental period of a periodic phenomenon pertaining to $v$-palindromes}
\author{Daniel Tsai\\
Graduate School of Mathematics, Nagoya University\\
Furocho, Chikusa-ku, Nagoya 464-8602}
\smallbreak \maketitle
\begin{abstract}
Natural numbers satisfying a certain unusual property are defined by the author in a previous note. Later, the author called such numbers $v$-palindromic numbers and proved a periodic phenomenon pertaining to such numbers and repeated concatenations of the digits of a number. It was left as a problem of further investigation to find the smallest period. In this paper, we provide a method to find the smallest period. Some theorems from signal processing are used, but we also supply our own proofs.
\end{abstract}


\vskip 0.2truein

\vskip 0.2truein
\setcounter{equation}{0}
\setcounter{section}{0}

\section{Introduction.}

  In \cite{T2}, natural numbers satisfying an unusual property are defined and their infinitude proved. Consider the natural number $56056$. The number formed by reversing its decimal digits is $65065$. Their canonical factorizations are
\begin{align}
  56056 &= 2^3\cdot 7^2\cdot11\cdot13, \label{example} \\
  65065 &= 5\cdot7\cdot11\cdot13^2. \label{reverse}
\end{align}
Notice that
\begin{equation}
  (2+3)+(7+2)+11+13=5+7+11+(13+2),
\end{equation}
which is a bit surprising. 
In other words, the sum of the prime divisors and exponents larger than $1$ on the right-hand side of \eqref{example} is equal to that of \eqref{reverse}. Such numbers are called $v$-palindromic numbers in \cite{T1}, of which we shall give a formal definition.

\begin{definition}
  Let $n$ be a natural number, its {\em reverse} is the number formed by reversing its decimal digits, denoted $r(n)$. Therefore $r(n)$ has the same number of digits as $n$ if $10\nmid n$ and fewer digits if $10\mid n$. (Here, a notation like $a\mid b$ means that $a$ divides $b$.)
\end{definition}

\begin{definition}
  For a natural number $n>1$, its {\em factorization sum} is the sum of the prime divisors and exponents larger than $1$ in its canonical factorization, denoted $v(n)$. Also, by convention, $v(1)=1$.
\end{definition}

The notation $v(n)$ is the one being used in both \cite{T2} and \cite{T1}, so we continue to use it here. The ``$v$" actually came from ``value". The quantity $v(n)$ is thought of as the ``value" of $n$. We obviously have the following.

\begin{theorem}
  The function $v\colon\mathbb{N}\to \mathbb{N}\cup\{0\}$ is additive. That is, $v(mn)=v(m)+v(n)$ whenever $m$ and $n$ are coprime natural numbers.
\end{theorem}

\begin{definition}
  A natural number $n$ is {\em $v$-palindromic} (or a {\em $v$-palindrome}) if $n$ is not a multiple of $10$, $n\ne r(n)$ (i.e.\ $n$ is not palindromic), and $v(n)=v(r(n))$.
\end{definition}

  We explain the choice of the name {\em $v$-palindrome}. A natural number $n$ is a {\em palindrome} if $n=r(n)$. The relation $v(n)=v(r(n))$ differs by having an $v$ in front. The condition that $n$ is not a multiple of $10$ is for ensuring that $r(n)$ does not have fewer digits than $n$. If we do not impose this condition, then $560$ would be $v$-palindromic, however we insist on imposing that $10\nmid n$ and therefore not consider $560$ as $v$-palindromic. The condition that $n\ne r(n)$ is included because if $n=r(n)$, then plainly $v(n)=v(r(n))$, which would not be surprising at all. Another possible viewpoint is to discard $n\ne r(n)$ in the definition of $v$-palindromes and to regard the palindromes as {\it trivial} $v$-palindromes. We do not adopt this alternative viewpoint though.
  
  The smallest $v$-palindrome is $18$, because $18=2\cdot 3^2$, $81=3^4$, and both factorization sums are $7$. The first natural question is then whether there are infinitely many of them, and the answer is affirmative. As proved in \cite{T2}, all the numbers
\begin{equation}\label{1stexample}
  18, 198, 1998, 19998, \ldots
\end{equation}
are $v$-palindromes. Also mentioned in \cite{T2}, there is another sequence of $v$-palindromes
\begin{equation}\label{2ndexample}
  18,1818,181818,\ldots,
\end{equation}
where we simply continue to concatenate $18$. It was this sequence which inspired the content of \cite{T1}, which investigates which of the repeated concatenations of a number are $v$-palindromes, and found a periodic phenomenon. It then posed three problems of further investigation pertaining to this periodic phenomenon. It is the purpose of this article to address the first two of these problems. In the next section, we shall recall the main theorem of \cite{T1}.

\section{Repeated concatenations and $\lowercase{v}$-palindromicity.}\label{sect2}

We first give the following notation.

\begin{definition}
  Let $n$ be a natural number, then the number formed by repeatedly concatenating its decimal digits $k$ times is denoted by $n(k)$.
\end{definition}

For example, $18(3)=181818$ and $56056(4)=56056560565605656056$. The main theorem of \cite{T1} can now be stated.

\begin{theorem}\label{main1}
  Let $n$ be a natural number such that $10\nmid n$ and $n\ne r(n)$. There exists an integer $\omega>0$ such that for all integers $k\ge1$, $n(k)$ is $v$-palindromic if and only if $n(k+\omega)$ is $v$-palindromic.
\end{theorem}

\begin{definition}\label{defnp}
  For $n$ as in Theorem \ref{main1}, a permissible $\omega$ will be called a {\em period} of $n$. The smallest one will be called the {\em fundamental period} of $n$, denoted $\omega_0(n)$.
\end{definition}

Regarding periods and the fundamental period, we have the following, which follows from Exercise 17(a) on p.\ 145 of \cite{A}.

\begin{theorem}
  Let $n$ be as in Theorem \ref{main1}, then the set of all periods is the set of all positive integral multiples of $\omega_0(n)$.
\end{theorem}

For example, since all the numbers \eqref{2ndexample} are $v$-palindromes, $\omega_0(18)=1$. In fact, $\omega_0(56056)=1$ too. The proof of Theorem \ref{main1} of \cite{T1} is constructive and found a particular period. In order to state this period, we need to define certain numbers which are introduced in \cite{T1} in Lemma 1.

\begin{definition}\label{hqd}
  Let $p^\alpha$ be a prime power, where $p\ne2,5$, and let $d$ be a natural number. Denote by $h_{p^\alpha,d}$ the order of $10^d$ regarded as an element of $(\mathbb{Z}/p^{\alpha+\ord_p(10^d-1)}\mathbb{Z})^\times$. (Here, $\ord_p(a)$ denotes the exponent of the prime $p$ in the canonical factorization of $a$.) In other words, $h_{p^\alpha,d}$ is the smallest positive integer such that
  \begin{equation}
    (10^d)^{h_{p^\alpha,d}}\equiv 1\pmod{p^{\alpha+\ord_p(10^d-1)}}.
  \end{equation}
  By Lemma 1 of \cite{T1}, $h_{p^\alpha,d}>1$.
\end{definition}

\begin{definition}\label{cp}
  Let $n$ be as in Theorem \ref{main1}. A {\it crucial prime} of $n$ is a prime $p$ for which $\ord_p(n)\ne\ord_p(r(n))$. The set of all crucial primes of $n$ will be denoted $K(n)$.
\end{definition}

The constructed period of $n$ in \cite{T1} is the following.

\begin{theorem}
  Let $n$ be as in Theorem \ref{main1} and let the number of decimal digits $n$ has be denoted $d$, then
  \begin{equation}\label{wf}
    \omega_f(n)=\lcm\{h_{p^2,d}\mid p\in K(n)\setminus\{2,5\}\}.
  \end{equation}
  is a period of $n$.
\end{theorem}

Here the ``$f$" in the notation $\omega_f(n)$ comes from ``found", because $\omega_f(n)$ is a period of $n$ found. The formula \eqref{wf} is originally written as
\begin{equation}\label{wfo}
  \omega_f(n)=\lcm\{h_{p,d},h_{p^2,d}\mid p\in K(n)\setminus\{2,5\}\}
\end{equation}
in \cite{T1}. However, 
in fact we always have $h_{p,d}\mid h_{p^2,d}$, thus \eqref{wfo} can be written more shortly as \eqref{wf}.
We show that in fact $h_{p,d}\mid h_{p^2,d}$. Since
\begin{equation}
  (10^d)^{h_{p^2,d}}\equiv 1\pmod{p^{2+\ord_p(10^d-1)}},
\end{equation}
plainly
\begin{equation}
  (10^d)^{h_{p^2,d}}\equiv 1\pmod{p^{1+\ord_p(10^d-1)}}.
\end{equation}
Now $h_{p,d}$ is the order of $10^d$ regarded as an element of $(\mathbb{Z}/p^{1+\ord_p(10^d-1)}\mathbb{Z})^\times$, thus $h_{p,d}\mid h_{p^2,d}$ follows from the structure of cyclic groups.

After calculating $\omega_0(n)$ and $\omega_f(n)$ for small $n$, the following is conjectured in \cite{T1}.

\begin{conjecture}\label{conj1}
  Let $n$ be as in Theorem \ref{main1}, then either $\omega_0(n)=1$ or $\omega_0(n)=\omega_f(n)$.
\end{conjecture}

It is a purpose of this article to provide a counterexample to Conjecture \ref{conj1}, thereby disproving it.

Another issue raised in \cite{T1} concerns whether given an $n$, there do exist a repeated concatenation of $n$ which is a $v$-palindrome. For $n=12$, no such repeated concatenation exists, i.e.\ all the numbers
\begin{equation}
  12,1212,121212,\ldots
\end{equation}
are not $v$-palindromic. For $n=13$ however, the first fourteen repeated concatenations are not $v$-palindromic but the fifteenth is. That is, $13(k)$ is not $v$-palindromic for $1\le k \le14$, but $13(15)$ is $v$-palindromic. Based on this phenomenon, the following definition is also given in \cite{T1}.

\begin{definition}\label{defnorder}
  Let $n$ be as in Theorem \ref{main1}. If there exists an integer $k\ge1$ such that $n(k)$ is $v$-palindromic, the smallest one will be called the {\em order} of $n$ and denoted $c(n)$. If no such $k$ exists then we write $c(n)=\infty$.
\end{definition}

It is posed as a problem of further investigation in \cite{T1} to find a simple way to determine whether, given an $n$, that $c(n)=\infty$ or not.

In this article, we provide a general procedure, starting with a given $n$ as in Theorem \ref{main1}, i.e.\ $n$ is a natural number, not a multiple of $10$, and not a palindrome. This procedure will determine whether $c(n)=\infty$ or not, and if not, determine both $\omega_0(n)$ and the precise conditions on $k\ge1$ such that $n(k)$ is $v$-palindromic. This procedure is mostly a realization of the proof of Theorem \ref{main1} in \cite{T1} into a more algorithmic nature.

For any given $n$ as in Theorem \ref{main1}, we shall construct in Section \ref{indicators} a function $I^n\colon\mathbb{Z}\to\{0,1\}$, which for positive integers, evaluates to $1$ if $n(k)$ is a $v$-palindrome and evaluates to $0$ otherwise. The superscript $n$ is only for specifying $n$ and does not denote composition of functions. Then, both $c(n)$ and $\omega_0(n)$ can be directly ``read off" from $I^n$ when it is expressed in a certain form. An important part of this paper is the proof of Theorem \ref{Thm12} using Theorem 12 in \cite{V1} (labeled as Theorem \ref{Theorem12} in this paper). As a corollary of Theorem \ref{Thm12}, $\omega_0(n)$ can be easily found from $I^n$ (Corollary \ref{findw0}).

We provide an appendix on the more general topic of periodic arithmetical functions. There, a formula for the fundamental period of an arbitrary periodic arithmetical function from $\mathbb{Z}$ to $\mathbb{C}$ is given (Theorem \ref{funperiod}). This formula is actually equivalent to the formula given in Theorem 9 in \cite{V1}. We prove their equivalence in Section \ref{5.2}. Although \cite{V1} contains a proof of Theorem \ref{Theorem12}, we provide a proof using Theorem \ref{funperiod} in Section \ref{12.3}. In this way, our paper is more self-contained.

In the field of signal processing, an arithmetical function from $\mathbb{Z}$ to $\mathbb{C}$ is called a {\em discrete signal} or {\em discrete-time signal} (see \cite{R}\cite{V1}\cite{V3}). We feel that it is better to include the appendix because our presentation differs from that in the signal processing context and should be interesting in its own right and perhaps in a language more familiar to number theorists.

\section{The functions $\varphi_{\lowercase{p},\delta}$.}

In this section we define certain functions which will be used later.

\begin{definition}
  For a prime $p$ and integer $\delta\ge2$, define the function
  \begin{equation}\label{defn1}
    \varphi_{p,\delta}(\alpha)=
    \begin{cases}
      p+\delta & \text{if $\alpha=0$,} \\
      1+\delta & \text{if $\alpha=1$,} \\
      \delta & \text{if $\alpha\ge2$.}
    \end{cases}
  \end{equation}
  For a prime $p\ne2$, define the function
  \begin{equation}\label{defn2}
    \varphi_{p,1}(\alpha)=
    \begin{cases}
      p & \text{if $\alpha=0$,} \\
      2 & \text{if $\alpha=1$,} \\
      1 & \text{if $\alpha\ge2$.}
    \end{cases}
  \end{equation}
  Finally, define
  \begin{equation}\label{defn3}
    \varphi_{2,1}(\alpha)=
    \begin{cases}
      2 & \text{if $\alpha=0,1$,} \\
      1 & \text{if $\alpha\ge2$.}
    \end{cases}
  \end{equation}
\end{definition}

  Hence we have, for any pair $(p,\delta)$ of a prime $p$ and natural number $\delta$, defined a function $\varphi_{p,\delta}\colon \mathbb{N}\cup\{0\}\to\mathbb{N}$. We give notations for their ranges as follows.

  \begin{definition}
    For $p$ a prime and $\delta$ a natural number, put $R_{p,\delta}=\varphi_{p,\delta}(\mathbb{N}\cup\{0\})$.
  \end{definition}
  
  Hence $|R_{p,\delta}|\in\{2,3\}$, being $2$ if and only if $(p,\delta)=(2,1)$. We have the following lemma.

\begin{lemma}\label{Cases}
  For an ordered quadruple $(p,\delta,u,\mu)$, where $p$ is a prime, $\delta$ a natural number, $u\in R_{p,\delta}$, $\mu\ge0$ an integer, exactly one of the following is the case.
  
  \renewcommand{\labelenumi}{[\roman{enumi}]}
  \begin{enumerate}
    \item $ \varphi^{-1}_{p,\delta}(u)=\{0\}$ and $\mu=0$, or $\varphi^{-1}_{p,\delta}(u)=\{1\}$ and $\mu=1$, or $\varphi^{-1}_{p,\delta}(u)=\{0,1\}$ and $\mu=1$, \label{item1}
      \item $\varphi^{-1}_{p,\delta}(u)=\{1\}$ and $\mu=0$,
      \item $\varphi^{-1}_{p,\delta}(u)=\{0,1\}$ and $\mu=0$,
      \item $\varphi^{-1}_{p,\delta}(u)=\mathbb{N}\setminus\{1\}$ and $\mu=1$,
      \item $\varphi^{-1}_{p,\delta}(u)=\mathbb{N}\setminus\{1\}$ and $\mu=0$,
      \item $\varphi^{-1}_{p,\delta}(u)=\mathbb{N}\setminus\{1\}$ and $\mu\ge2$,
      \item otherwise.
  \end{enumerate}
\end{lemma}
\begin{proof}
  In view of \eqref{defn1}, \eqref{defn2}, and \eqref{defn3}, $\varphi_{p,\delta}$ is one of $\{0\}$, $\{1\}$, $\{0,1\}$, and $\mathbb{N}\setminus\{1\}$. Then one sees that the first six cases are mutually exclusive. That each case is possible is plain.
\end{proof}

\begin{definition}\label{defnD}
  For each quadruple $(p,\delta,u,\mu)$ as in Lemma \ref{Cases}, denote by $D(p,\delta,u,\mu)$ the case number (in lower case Roman numerals in brackets as in the lemma). That is, $D(p,\delta,u,\mu)=[\mathrm{ii}]$ if and only if $\varphi^{-1}_{p,\delta}(u)=\{1\}$ and $\mu=0$, $D(p,\delta,u,\mu)=[\mathrm{iii}]$ if and only if $\varphi^{-1}_{p,\delta}(u)=\{0,1\}$ and $\mu=0$, etc.
\end{definition}    

\section{The conditions when a repeated concatenation is a $\lowercase{v}$-palindrome.}\label{rules}

Throughout this section we fix a natural number $n$ as in Theorem \ref{main1}, i.e.\ $n$ is a natural number, not a multiple of $10$, and not a palindrome. We find the precise conditions when the number formed by repeatedly concatenating $k$ times the decimal digits of $n$, i.e.\ $n(k)$, is a $v$-palindrome. Almost no proofs are given because they all follow from \cite{T1}.
  
  Suppose that $n$ and $r(n)$ have the following canonical factorizations.
\begin{align}
  n&= \prod_{p} p^{a_p}, \\
  r(n) &=\prod_p p^{b_p},
\end{align}  
where the products are over the primes, the $a_p,b_p\ge0$ are integers, and $a_p=b_p=0$ for all but finitely many primes $p$. Let the number of decimal digits of $n$ be denoted $d$. We give the following notation.
  \begin{definition}
    For $k\ge1$, put 
    \begin{equation}
    \rho_{k}=\overbrace{
1\underbrace{0\ldots0}_\text{$d-1$}1
\underbrace{0\ldots0}_\text{$d-1$}1
\ldots
1\underbrace{0\ldots0}_\text{$d-1$}1
}^\text{$k$}.
  \end{equation}\label{rhok}
  That is, we have $k$ ones and in between any two consecutive ones, $d-1$ zeros.
  \end{definition}
  We consider $n$ as fixed and therefore $d$ is also fixed. The $k$ in \eqref{rhok} is considered a variable and denotes the number of times we repeatedly concatenate the digits of $n$. We shall describe the necessary and sufficient condition on $k$ such that $n(k)$ is a $v$-palindrome. Because of the way $\rho_k$ is defined, we obviously have the following.
  \begin{lemma}
    For every $k\ge1$, $n(k)=n\rho_k$.
  \end{lemma}
  We shall determine a complete set of mutually exclusive conditions on $k$ such that $n(k)$ is a $v$-palindrome if and only if $k$ satisfies one of those conditions.
  
  For each crucial prime $p$ of $n$ (Definition \ref{cp}), put
  \begin{align}
\delta_p &= a_p-b_p\ne0,\\
\mu_p &= \min(a_p,b_p)\ge0, \\
g_p & =\ord_p(\rho_k), \\
\alpha_p &= \mu_p+g_p,
\end{align}
where $\delta_p\ne0$ by the definition of crucial prime. The $\delta_p$ and $\mu_p$ depends only on $n$, thus can be considered as fixed. The $g_p$ clearly depends on not only $p$ but also $k$. However, we omit $k$ from the notation for simplicity, keeping in mind that $g_p$ depends also on the variable $k$. Consequently, $\alpha_p$ also depends on $k$. In describing the conditions for which $k$ must satisfy for $n(k)$ to be a $v$-palindrome, $\alpha_p$ would occur. We shall simply denote the set of crucial primes of $n$ as $K$, which is a nonempty finite set of prime numbers.
  
\begin{definition}
The equation
  \begin{equation}\label{chareq}
    \sum_{p\in K}\sgn(\delta_p)u_p=0,
  \end{equation}
  where $\sgn$ is the sign function with $\operatorname{sgn}(\delta_p)=1$ if $\delta_p>0$ and $\sgn(\delta_p)=-1$ if $\delta_p<0$, will be called the {\em characteristic equation} for $n$, where the $u_p$ are variables.
\end{definition}

We want to solve \eqref{chareq} for the $u_p$ but with certain restrictions.

  \begin{definition}
    A solution $(u_p)_{p\in K}$ to the characteristic equation \eqref{chareq} with $u_p\in R_{p,|\delta_p|}$ for all $p\in K$ will be called a {\it characteristic solution} for $n$. The set of all characteristic solutions will be denoted by $U$.
  \end{definition}
  
     Since the number of digits of $n$ is denoted $d$, we have the numbers $h_{q,d}$ defined in Definition \ref{hqd} for every prime power $q$, relatively prime to $10$. We shall omit the $d$ and simply write $h_{q}$. Let $\mathbf{u}=(u_p)_{p\in K}$ be a characteristic solution for $n$. We denote, for $p\in K\setminus\{2,5\}$,
    \begin{equation}\label{defnT1}
      T_{p,\mathbf{u}}=(A_{p,\mathbf{u}},B_{p,\mathbf{u}})=
      \begin{cases}
        (\varnothing,\{h_{p}\})& \text{if $D(p,|\delta_p|,u_p,\mu_p)=[\mathrm{i}]$,} \\
        (\{h_{p}\},\{h_{p^2}\}) & \text{\text{if $D(p,|\delta_p|,u_p,\mu_p)=[\mathrm{ii}]$,}} \\
        (\varnothing,\{h_{p^2}\}) & \text{\text{if $D(p,|\delta_p|,u_p,\mu_p)=[\mathrm{iii}]$,}} \\
       (\{h_{p}\},\varnothing) & \text{\text{if $D(p,|\delta_p|,u_p,\mu_p)=[\mathrm{iv}],$}} \\
        (\{h_{p^2}\},\varnothing)& \text{\text{if $D(p,|\delta_p|,u_p,\mu_p)=[\mathrm{v}].$}}
      \end{cases}
    \end{equation}
    For $p\in\{2,5\}$, denote
    \begin{equation}\label{defnT2}
       T_{p,\mathbf{u}}=(A_{p,\mathbf{u}},B_{p,\mathbf{u}})=
      \begin{cases}
        (\varnothing,\varnothing)& \text{if $D(p,|\delta_p|,u_p,\mu_p)=[\mathrm{i}]$,} \\
        (\varnothing,\{1\}) & \text{\text{if $D(p,|\delta_p|,u_p,\mu_p)=[\mathrm{ii}]$,}} \\
        (\varnothing,\varnothing) & \text{\text{if $D(p,|\delta_p|,u_p,\mu_p)=[\mathrm{iii}],$}} \\
       (\varnothing,\{1\}) & \text{\text{if $D(p,|\delta_p|,u_p,\mu_p)=[\mathrm{iv}],$}} \\
        (\varnothing,\{1\})& \text{\text{if $D(p,|\delta_p|,u_p,\mu_p)=[\mathrm{v}]$.}}
      \end{cases}
    \end{equation}
    Also, we denote, for any $p\in K$,
    \begin{equation}\label{defnT3}
    T_{p,\mathbf{u}}=(A_{p,\mathbf{u}},B_{p,\mathbf{u}})=
    \begin{cases}
      (\varnothing,\varnothing) & \text{if $D(p,|\delta_p|,u_p,\mu_p)=[\mathrm{vi}]$,} \\
      (\varnothing,\{1\}) & \text{if $D(p,|\delta_p|,u_p,\mu_p)=[\mathrm{vii}]$.}
    \end{cases}
    \end{equation}
    
Therefore $T_{p,\mathbf{u}}$ is an ordered pair of sets of at most one positive integer. The ``$T$" comes from ``table", because these ordered pairs can be arranged into the form of a table of $p$ versus $\mathbf{u}$, with entries $T_{p,\mathbf{u}}$, which might be easier to comprehend in practice. We give the following general notation.
    \begin{definition}\label{SAB}
      Let $A$ and $B$ be finite sets of positive integers, then denote
      \begin{equation}
        S(A,B)=\{x\in\mathbb{Z}\mid (\text{for all } a\in A, a\mid x)\text{ and }(\text{for all } b\in B, b\nmid x)\}.
      \end{equation}
      That is, $S(A,B)$ is the set of all integers divisible by every element of $A$, but indivisible by every element of $B$.
    \end{definition}
    \begin{definition}\label{AuBu}
      For each characteristic solution $\mathbf{u}$ for $n$, put
      \begin{equation}
        A_{\mathbf{u}}=\bigcup_{p\in K}A_{p,\mathbf{u}},\quad B_{\mathbf{u}}=\bigcup_{p\in K}B_{p,\mathbf{u}},
      \end{equation}
      and $S_{\mathbf{u}}=S(A_{\mathbf{u}},B_{\mathbf{u}})$.  
    \end{definition}
    
    The first sentence in the following theorem is Lemma 4 in \cite{T1}, and the second sentence follows from arguments following Lemma 4 in \cite{T1}.
    
    \begin{theorem}\label{Thm5}
  For $k\ge1$, the number $n(k)$ is a $v$-palindrome if and only if for some characteristic solution $\mathbf{u}=(u_p)_{p\in K}$ for $n$,
  \begin{equation}\label{equalities}
    \varphi_{p,|\delta_p|}(\alpha_p)=u_p,\quad \text{for all } p\in K.
  \end{equation}
  Moreover, given a characteristic solution $\mathbf{u}=(u_p)_{p\in K}$ for $n$, \eqref{equalities} holds if and only if
  \begin{equation}\label{sets}
    k\in S_{\mathbf{u}}.
  \end{equation}
\end{theorem}

The condition \eqref{equalities} might seem to be independent of $k$, but if we recall, the $\alpha_p$ actually depends on $k$. To write \eqref{equalities} out so that the dependence on $k$ is more visible, we can recover \eqref{equalities} into
\begin{equation}\label{equalitiess}
  \varphi_{p,|\delta_p|}(\mu_p+\ord_p(\rho_k))=u_p,\quad \text{for all } p\in K.
\end{equation}
Since the condition \eqref{equalities} (or equivalently \eqref{equalitiess}) cannot hold, for the same $k$, for two distinct characteristic solutions, the conditions \eqref{equalities} are mutually exclusive over $\mathbf{u}$. Consequently, the conditions \eqref{sets} are also mutually exclusive over $\mathbf{u}$. Therefore the sets $S_{\mathbf{u}}\cap\mathbb{N}$ are pairwise disjoint, we write this as a corollary.

\begin{corollary}\label{forN}
  The sets $S_\mathbf{u}\cap\mathbb{N}$ are pairwise disjoint over $\mathbf{u}\in U$.
\end{corollary}
 In fact we have the following, which says that not only are the intersections $S_\mathbf{u}\cap\mathbb{N}$ of the sets $S_\mathbf{u}$ with $\mathbb{N}$ pairwise disjoint, but the sets $S_\mathbf{u}$ themselves are already pairwise disjoint as subsets of $\mathbb{Z}$.

\begin{theorem}\label{forNZ}
  The sets $S_\mathbf{u}$ are pairwise disjoint over $\mathbf{u}\in U$.
\end{theorem}
\begin{proof}
  Suppose on the contrary that for some distinct $\mathbf{u},\mathbf{v}\in U$ that there exists an integer
  \begin{equation}
    x\in S_\mathbf{u}\cap S_{\mathbf{v}}=S(A_\mathbf{u},B_{\mathbf{u}})\cap S(A_\mathbf{v},B_{\mathbf{v}}).
  \end{equation} If we let
  \begin{equation}
    \omega=\lcm(A_{\mathbf{u}}\cup B_{\mathbf{u}}\cup A_{\mathbf{v}}\cup B_{\mathbf{v}}),
  \end{equation}
  then we see that $x+\omega\in S_\mathbf{u}\cap S_{\mathbf{v}}$ too. Therefore adding $\omega$ as many times as necessary to $x$, we obtain a natural number in $S_\mathbf{u}\cap S_{\mathbf{v}}$, this contradicts Corollary \ref{forN}.
\end{proof}

\begin{corollary}\label{represent}
  The set of all $k\ge1$ such that $n(k)$ is a $v$-palindrome is
  \begin{equation}
    \bigsqcup_{\mathbf{u}\in U}(S_{\mathbf{u}}\cap\mathbb{N})=\left(\bigsqcup_{\mathbf{u}\in U} S_{\mathbf{u}}\right)\cap\mathbb{N}.
  \end{equation}
\end{corollary}
\begin{proof}
  This follows directly from Theorem \ref{Thm5} and Corollaries \ref{forN} and \ref{forNZ}.
\end{proof}

  Thus $k$ can be categorized as to which $S_{\mathbf{u}}$ it belongs to. However, it could happen that $S_{\mathbf{u}}=\varnothing$, therefore we give the following definition.

    \begin{definition}
      If $S_{\mathbf{u}}$ is empty, then we call $\mathbf{u}$ a {\em degenerate} characteristic solution for $n$, otherwise it is {\em nondegenerate}. The set of all nondegenerate characteristic solutions will be denoted by $U^\ast$. For a $\mathbf{u}\in U^\ast$, an $n(k)$ which is a $v$-palindrome will be said to be of type $\mathbf{u}$ (with respect to $n$) if $k\in S_\mathbf{u}$. We also denote
      \begin{equation}\label{disjoint}
        S=\bigsqcup_{\mathbf{u}\in U^\ast} S_{\mathbf{u}}.
      \end{equation}
    \end{definition}
    
    We have included ``with respect to $n$" in our definition of type above because the same $v$-palindrome $m$ might be $m=n_1(k_1)=n_2(k_2)$ for $n_1\ne n_2$, and therefore the type of $m$ can be considered with respect to $n_1$ and also with respect to $n_2$. Whether the notion of type defined above is really dependent on $n$ or not is still unclear. For instance, if we consider the $v$-palindrome $m=13(15)$, then $m=13(3)(5)=13(5)(3)=13(15)(1)$ also. Perhaps a bit surprisingly, in all four cases the type of $m$ is $(2,2)$, or more precisely, the $(u_p)_{p\in\{13,31\}}$ with $u_{13}=u_{31}=2$. Thus we make the following conjecture.
    
    \begin{conjecture}\label{invariant}
      Let $m$ be a $v$-palindrome such that $m=n_1(k_1)=n_2(k_2)$, then the type of $m$ with respect to $n_1$ is the same as the type of $m$ with respect to $n_2$. 
    \end{conjecture}
    
We shall omit saying ``with respect to $n$" hereafter, it being understood implicitly, though if Conjecture \ref{invariant} were true, then omitting ``with respect to $n$" would be completely appropriate.
    
    Notice that if $\mathbf{u}$ is nondegenerate, then there exists a $v$-palindrome $n(k)$ of type $\mathbf{u}$, because $S_\mathbf{u}$ contains positive integers. We have thus categorized the $v$-palindromes $n(k)$ into a number of types which equals the number of nondegenerate characteristic solutions. For the characteristic equation \eqref{chareq} for $n$, it could happen that there are no characteristic solutions at all, or that there are characteristic solutions but unfortunately all are degenerate, or that there are nondegenerate solutions. In the former two cases $n(k)$ is not a $v$-palindrome, for any $k\ge1$, i.e.\ $c(n)=\infty$. In the third case only does there exist a $k\ge1$ for which $n(k)$ is a $v$-palindrome.

  Let us summarize this section. We started with a natural number $n$, not a multiple of $10$, and not a palindrome. We have the sequence of repeated concatenations of the decimal digits of $n$, namely $n(k)$ ($k\ge1$). We would like to know which of them are $v$-palindromes. In order to do this, We first solve for the characteristic solutions for $n$. Then, each nondegenerate characteristic solution $\mathbf{u}$ gives rise to a nonempty infinite subset $S_{\mathbf{u}}\cap\mathbb{N}$ of integers $k\ge1$ for which $n(k)$ is a $v$-palindrome. The sets $S_{\mathbf{u}}\cap\mathbb{N}$ are pairwise disjoint over the nondegenerate solutions $\mathbf{u}$ and their union gives the set of all $k\ge1$ for which $n(k)$ is a $v$-palindrome. This section is of a more theoretical and abstract nature, Section \ref{procedure} puts these ideas into a more algorithmic description.

  \section{The indicator function.}\label{indicators}
  
Again, throughout this section we fix a natural number $n$ as in Theorem \ref{main1}, i.e.\ $n$ is a natural number, not a multiple of $10$, and not a palindrome. In Section \ref{rules}, the set of all $k\ge1$ for which $n(k)$ is a $v$-palindrome is represented as a disjoint union of sets in Corollary \ref{represent}. In this section, we construct a function $I(k)$ which evaluates to $1$ if $n(k)$ is a $v$-palindrome and $0$ if not.

\subsection{Definition of the indicator funcion.}
  
  Recall that an indicator function is defined as follows.
  \begin{definition}
    If $A\subseteq\Omega$, then the {\em indicator function} of $A$ in $\Omega$ is the function $I_A\colon \Omega\to\{0,1\}$ defined by
    \begin{equation}
      I_A(x)=\begin{cases}
        1 & \text{if $x\in A$,} \\
        0 & \text{if $x\in \Omega\setminus A$.}
      \end{cases}
    \end{equation}
    In particular, for an integer $a\ge1$, we denote the indicator function of $a\mathbb{Z}\subseteq\mathbb{Z}$ by $I_a$. That is, for $x\in\mathbb{Z}$,
    \begin{equation}
      I_a(x)=
      \begin{cases}
        1 & \text{if $a\mid x$,} \\
        0 & \text{if $a\nmid x$.}
      \end{cases}
    \end{equation}
  \end{definition}

We have the following representation of the indicator function of the set $S(A,B)$ defined in Definition \ref{SAB}.
\begin{lemma}
  Let $A,B\subseteq\mathbb{N}$ be finite sets, then for all $x\in\mathbb{Z}$,
  \begin{equation}\label{decomp}
    I_{S(A,B)}(x)=I_{\lcm(A)}(x)\prod_{b\in B}(1-I_b(x)).
  \end{equation}
  Hence $I_{S(A,B)}$ is periodic modulo $\lcm(A\cup B)$.
\end{lemma}
\begin{proof}
  If $x\in S(A,B)$, then $a\mid x$ for all $a\in A$, and so $\lcm(A)\mid x$. Thus $I_{\lcm(A)}(x)=1$. Moreover, for every $b\in B$, $b\nmid x$, and so $I_b(x)=0$. Hence we see that the right-hand side of \eqref{decomp} is $1$.
  
  On the other hand, assume that $x$ is an integer with $x\notin S(A,B)$. Then either $x$ is not divisible by some particular $a\in A$, or is divisible by some particular $b\in B$. In the first case, $x$ cannot be a multiple of $\lcm(A)$, thus $I_{\lcm(A)}(x)=0$ and we see that the right-hand side of \eqref{decomp} is $0$. In the second case, $I_b(x)=1$ for some $b\in B$, hence one of the factors in the product on the right-hand side of \eqref{decomp} becomes $0$, and we see again that the right-hand side of \eqref{decomp} evaluates to $0$. This proves \eqref{decomp}.
  
  To prove the periodicity, we see that if we add to $x$ the quantity $\lcm({A\cup B})$, the values of $I_{\lcm(A)}$, and all the $I_b$ ($b\in B$) do not change, and hence $I_{S(A,B)}$ is periodic modulo $\lcm(A\cup B)$.
\end{proof}
Consequently, we directly have the following.

\begin{corollary}\label{indicatoru}
  Let $\mathbf{u}$ be a nondegenerate characteristic solution for $n$, then for all $x\in\mathbb{Z}$,
  \begin{equation}\label{Iu}
    I_{S_{\mathbf{u}}}(x)=I_{\operatorname{lcm}(A_{\mathbf{u}})}(x)\prod_{b\in B_{\mathbf{u}}}(1-I_b(x)).
  \end{equation}
  Hence $I_{S_\mathbf{u}}$ is periodic modulo $\lcm(A_\mathbf{u}\cup B_{\mathbf{u}})$. Moreover, for $k\ge1$, $n(k)$ is $v$-palindromic of type $\mathbf{u}$ if and only if $I_{S_{\mathbf{u}}}(k)=1$.
\end{corollary}

  Since we have the disjoint union \eqref{disjoint}, we have the following.

\begin{theorem}\label{indicator}
  We have that for all $x\in\mathbb{Z}$,
  \begin{equation}\label{decompS}
    I_S(x)=\sum_{\mathbf{u}\in U^\ast}I_{S_{\mathbf{u}}}(x).
  \end{equation}
  Hence $I_S$ is periodic modulo
  \begin{equation}\label{p}
    \lcm\left(\bigcup_{\mathbf{u}\in U^\ast}(A_{\mathbf{u}}\cup B_{\mathbf{u}})\right).  
  \end{equation}
  Moreover, for $k\ge1$, $n(k)$ is $v$-palindromic if and only if $I_S(k)=1$.
\end{theorem}
\begin{proof}
  If $x\in S$, then $x\in S_{\mathbf{u}}$ for exactly one $\mathbf{u}\in U^\ast$, so the right-hand side of \eqref{decompS} evaluates to $1$.  If $x$ is an integer with $x\notin S$, then $I_{S_{\mathbf{u}}}(x)=0$ for all $\mathbf{u}\in U^\ast$, so the right-hand side of \eqref{decompS} evaluates to $0$.   
  Let the quantity in \eqref{p} be denoted by $\omega$. For each $\mathbf{u}\in U^\ast$, $I_{S_\mathbf{u}}=I_{S(A_\mathbf{u},B_\mathbf{u})}$ is periodic modulo $\lcm(A_\mathbf{u}\cup B_\mathbf{u})$ by Corollary \ref{indicatoru}. Since $\omega$ is a multiple of $\lcm(A_\mathbf{u}\cup B_\mathbf{u})$, for every $\mathbf{u}\in U^\ast$, we see that $I_S$ is periodic modulo $\omega$.

\end{proof}

\begin{definition}
  The function $I_S$ in Theorem \ref{indicator} will be called the {\em indicator function} for $n$, and denoted simply as $I$, or if we want to indicate the $n$, denoted by $I^n$.
\end{definition}

\begin{theorem}\label{becomeI}
  Let $\omega>0$, then $\omega$ is a period of $I$ if and only if it is a period of $I|_\mathbb{N}$ if and only if it is a period of $n$ (in the sense of Definition \ref{defnp}). Hence $\omega_0(n)$ is the fundamental period of $I$.
\end{theorem}
\begin{proof}
  This follows from Theorems \ref{restriction} and \ref{indicator}, and the definition (Definition \ref{defnp}) of a period of a number $n$.
\end{proof}

Because of the above theorem, to find $\omega_0(n)$, we just have to find the fundamental period of $I$. If we try to do this by using Theorem \ref{funperiod}, then we will have to first express $I$ into the form \eqref{f}. This is doable but conceivably tedious. Instead, there is a much easier way, accomplished by writing $I$ into a linear combination of functions of the form $I_a$ with integer coefficients (Theorem \ref{incsum}), which we discuss in the next subsection.

\subsection{The indicator function as a linear combination.}

The following lemma is what is used to write the indicator function $I$ as a linear combination of functions of the form $I_a$ with integer coefficients.

\begin{lemma}\label{comb}
  For any integers $a,b\ge1$, $I_aI_b=I_{\lcm(a,b)}$.
\end{lemma}
\begin{proof}
  We need to prove that for all $x\in \mathbb{Z}$,
  \begin{equation}
    I_a(x)I_b(x)=I_{\lcm(a,b)}(x).
  \end{equation}
  If $\lcm(a,b)\mid x$, then both $a\mid x$ and $b\mid x$, thus both sides of the above equation evaluates to $1$. If $\lcm(a,b)\nmid x$, then either $a\nmid x$ or $b\nmid x$, thus in the above equation, one of the factors on the left-hand side is $0$, and the right-hand side is also $0$. This completes the proof.
\end{proof}

\begin{theorem}\label{Thm13}
  Let $\mathbf{u}$ be a nondegenerate characteristic solution for $n$, then
  \begin{equation}
    I_{S_{\mathbf{u}}}=\sum_{B\subseteq B_{\mathbf{u}}}(-1)^{|B|}I_{\lcm(A_{\mathbf{u}}\cup B)}.
  \end{equation}
\end{theorem}
\begin{proof}
  This follows by expanding the equation \eqref{Iu} in Corollary \ref{indicatoru} and then simplifying using Lemma \ref{comb}.
\end{proof}

Similarly, we have the following.

\begin{theorem}\label{expthm}
  We have the expression
  \begin{equation}\label{exp}
    I=\sum_{\mathbf{u}\in U^\ast}\sum_{B\subseteq B_{\mathbf{u}}}(-1)^{|B|}I_{\lcm(A_{\mathbf{u}}\cup B)}
  \end{equation}
  for the indicator function for $n$.
\end{theorem}
\begin{proof}
  This follows from Theorem \ref{indicator} and Theorem \ref{Thm13}.
\end{proof}

We consequently have the following.

\begin{theorem}\label{incsum}
  There exist integers $0<c_1<c_2<\cdots<c_q$ and integers $\lambda_1,\lambda_2,\ldots,\lambda_q\ne0$ such that
  \begin{equation}\label{incsume}
    I=\sum^q_{j=1}\lambda_j I_{c_j},
  \end{equation}
  possibly with $q=0$, i.e.\ we have an empty sum.
\end{theorem}
\begin{proof}
  We simply collect like terms in equation \eqref{exp} in Theorem \ref{expthm}.
\end{proof}

We state without proof the following theorem, which can be proved by induction.

\begin{theorem}
If a function $f\colon\mathbb{Z}\to\mathbb{C}$ is represented as
\begin{equation}\label{form}
  f=\sum^q_{j=1}\lambda_j I_{c_j},
\end{equation}
where $c_1<\cdots<c_q$ are positive integers (possibly $q=0$) and $\lambda_1,\ldots,\lambda_q\ne 0$ are any integers, then this representation is unique, in the sense that if $c'_1<\cdots<c'_{q'}$ are positive integers and $\lambda'_1,\ldots,\lambda'_{q'}\ne0$ integers which satisfy
\begin{equation}
  f=\sum^{q'}_{j=1}\lambda'_j I_{c'_j},
\end{equation}
then $q=q'$ and for all $1\le j\le q$, $c_j=c'_j$ and $\lambda_j=\lambda'_j$.
\end{theorem}

According to the above theorem, we have in particular that the indicator function $I$ for $n$ can be expressed in the form \eqref{incsume} uniquely. Examples of some indicator functions are given in Table \ref{table:indicatorfunA}.

\section{Finding the fundamental period.}

Again we fix a natural number $n$ as in Theorem \ref{main1}. Our goal is to find the fundamental period $\omega_0(n)$. If we try to do this by using Theorem \ref{funperiod}, then we will have to first express the indicator function $I$ for $n$ into the form \eqref{f}. We explicitly express $I$ into this form in Section \ref{express}. However, this is not a smart way to find $\omega_0(n)$. Instead, in Section \ref{smart}, we show that $\omega_0(n)$ is simply the least common multiple of the $c_j$'s in Theorem \ref{incsum}, using a theorem in \cite{V1}.

\subsection{The indicator function in the form of \eqref{f}.}\label{express}

According to Theorem 8.1 on p.\ 158 in \cite{A}, we have the following representation of $I_a$ in terms of the $a$-th roots of unity in $\mathbb{C}$. Let the set of all $a$-th roots of unity in $\mathbb{C}$ be denoted by $R(a)$.
\begin{lemma}\label{indicatora}
  For $a\ge1$, we have that for all $x\in\mathbb{Z}$,
  \begin{equation}
    I_a(x)=\frac{1}{a}\sum_{\zeta\in R(a)}\zeta^x.
  \end{equation}
\end{lemma}

If we use the above lemma into the equation \eqref{incsume} in Theorem \ref{incsum}, we can express the indicator function into the form \eqref{f} (see also theorem below). In principle, in view of Theorem \ref{becomeI}, we can use Theorem \ref{funperiod} to calculate the fundamental period of $I$, which will then be $\omega_0(n)$. However, as aforementioned, this is not a smart way.

\begin{theorem}
  The indicator function is
  \begin{gather}
    I(x)=\sum_{\mathbf{u}\in U^\ast}\sum_{B\subseteq B_{\mathbf{u}}}\frac{(-1)^{|B|}}{\lcm(A_{\mathbf{u}}\cup B)}\sum_{\zeta\in R(\lcm(A_{\mathbf{u}}\cup B))}\zeta^x \label{one}\\
    =\sum_{\zeta\in R(\omega)} \left(\sum_{\mathbf{u}\in U^\ast, B\subset B_\mathbf{u},\zeta\in R(\lcm(A_{\mathbf{u}}\cup B))}\frac{(-1)^{|B|}}{\lcm(A_{\mathbf{u}}\cup B)}\right)\zeta^x,\label{two}
  \end{gather}
  where $\omega$ is the quantity given by \eqref{p} in Theorem \ref{indicator}.
\end{theorem}
\begin{proof}
  The first equality follows by using Lemma \ref{indicatora} into the equation \eqref{exp} in Theorem \ref{expthm}. The second equality is simply an iterated version of the first, summing over $\zeta$ first.
\end{proof}

\subsection{Finding the fundamental period from \eqref{incsume}}\label{smart}

We first give some definitions, more or less equivalent to some definitions given in \cite{V1}. The set of all functions $f\colon \mathbb{Z}\to\mathbb{C}$, which can be denoted by $\mathbb{C}^\mathbb{Z}$, is obviously a vector space over $\mathbb{C}$ with both vector addition and scalar multiplication defined pointwise. The set, denoted by $\mathcal{F}$ in Section \ref{5.1}, of all periodic arithmetical functions, is a subspace of $\mathbb{C}^\mathbb{Z}$. The so-called Ramanujan spaces can be defined as follows.

\begin{definition}
  Let $\omega\ge1$ be an integer. The set of all functions
  \begin{equation}
    f(x)=\sum_{\zeta\in R^\ast(\omega)}g(\zeta)\zeta^x,\quad\text{for }x\in\mathbb{Z},
  \end{equation}
  where the $g(\zeta)$'s are complex numbers and $R^\ast(\omega)$ denotes the set of primitive $\omega$-th roots of unity in $\mathbb{C}$, is a subspace of $\mathcal{F}$ called a {\it Ramanujan space} and denoted by $S_\omega$.
\end{definition}

Then, Theorem 12 in \cite{V1} can be stated as follows.

\begin{theorem}\label{Theorem12}
  Let $\omega_1,\ldots,\omega_m\ge1$ be distinct integers, and let $0\ne f_j\in S_{\omega_j}$ for each $1\le j\le m$ (the $0$ here denoting the zero function). Then the fundamental period of $f=f_1+\cdots+f_m$ is $\lcm(\omega_1,\ldots,\omega_m)$.
\end{theorem}

We use the above theorem to show that the fundamental period of a function in the form \eqref{form} is the least common multiple of the $c_j$'s.

\begin{theorem}\label{Thm12}
For a function of the form
  \begin{equation}
  f=\sum^q_{j=1}\lambda_j I_{c_j},
\end{equation}
where the $c_1<\cdots<c_q$ are positive integers (possibly $q=0$) and $\lambda_1,\ldots,\lambda_q\ne 0$ are any integers, its fundamental period is $\lcm(c_1,\ldots,c_q)$.
\end{theorem}
\begin{proof}
  Define the set
  \begin{equation}
    D=\{d\in\mathbb{N}\mid d\mid c_j\text{ for some }1\le j\le q\}.
  \end{equation}
  That is, $D$ is the union of the divisors of $c_1,\ldots,c_q$. In view of Lemma \ref{indicatora}, the function $f$ can be written as
  \begin{equation}
    f(x)=\sum_{d\in D}\sum_{\zeta\in R^\ast(d)}\left(\sum_{1\le j\le q, d\mid c_j}\frac{\lambda_j}{c_j}\right)\zeta^x
  \end{equation}
  Let us denote, for $d\in D$,
  \begin{equation}
  f_d(x)=\sum_{\zeta\in R^\ast(d)}\left(\sum_{1\le j\le q, d\mid c_j}\frac{\lambda_j}{c_j}\right)\zeta^x,
  \end{equation}
  so that $f_d\in S_{d}$. Then $f=\sum_{d\in D}f_d$. In view of Theorem \ref{Theorem12}, the fundamental period of $f$ is
  \begin{equation}
    \lcm\{d\in D\mid f_d\ne 0\}=\lcm\left\{d\in D\mid \sum_{1\le j\le q, d\mid c_j}\frac{\lambda_j}{c_j}\ne0\right\}.
  \end{equation}
  We have to show that
  \begin{equation}\label{show}
    \lcm\left\{d\in D\mid \sum_{1\le j\le q, d\mid c_j}\frac{\lambda_j}{c_j}\ne0\right\}=\lcm(c_1,\ldots,c_q).
  \end{equation}
  That the right-hand side above, denote it by $R$ (not the set of roots of unity defined in Section \ref{5.1}) is a multiple of the left-hand side, denote it by $L$, is plain. For each $d\in D$, we can write
  \begin{equation}
    \sum_{1\le j\le q, d\mid c_j}\frac{\lambda_j}{c_j}=\frac{\lambda_1}{c_1}[d\mid c_1]+\cdots+\frac{\lambda_q}{c_q}[d\mid c_q],
  \end{equation}
  where $[\cdot]$ is the Iverson bracket with $[P]=1$ if $P$ is true and $[P]=0$ if $P$ is false. Now suppose that $p^\alpha$ is any prime power with $p^\alpha\mid R$ but $p^{\alpha+1}\nmid R$. Let $j_0$ be the largest integer with $1\le j_0\le q$ and $p^\alpha\mid c_{j_0}$. Then
  \begin{equation}
  \sum_{1\le j\le q, c_{j_0}\mid c_{j}}\frac{\lambda_j}{c_j}=\frac{\lambda_1}{c_1}[c_{j_0}\mid c_1]+\cdots+\frac{\lambda_{j_0}}{c_{j_0}}[c_{j_0}\mid c_{j_0}]+\cdots+\frac{\lambda_q}{c_q}[c_{j_0}\mid c_q]=\frac{\lambda_{j_0}}{c_{j_0}}\ne0.
  \end{equation}
  This holds because, for $1\le j<j_0$, as $c_j<c_{j_0}$, plainly $[c_{j_0}\mid c_j]=0$; and for $j_0<j\le q$, if $c_{j_0}\mid c_j$, then $p^\alpha\mid c_j$, which contradicts our choice of $j_0$, thus $[c_{j_0}\mid c_j]=0$. Therefore as $L$ is a multiple of $c_{j_0}$, it is also a multiple of $p^\alpha$. Consequently, as $L$ is a multiple of every prime power divisor of $R$, $R\mid L$. Since both $L\mid R$ and $R\mid L$, \eqref{show} holds.
\end{proof}

As a consequence of the above theorem, we have the following corollary.

\begin{corollary}\label{findw0}
 Suppose that the indicator function for $n$ is expressed as
 \begin{equation}
   I=\sum^q_{j=1}\lambda_j I_{c_j},
 \end{equation}
 where $q\ge0$, $0<c_1<\ldots<c_q$, and $\lambda_1,\ldots,\lambda_q\ne0$ are integers. Then the fundamental period of $n$ is $\omega_0(n)=\lcm(c_1,\ldots,c_q)$.
\end{corollary}

\section{Finding the order.}

  We have defined the order $c(n)$ of a number $n$ in Definition \ref{defnorder}. It is the smallest integer $k\ge1$ such that $n(k)$ is a $v$-palindrome, if such a $k$ exists, and is $\infty$ otherwise. After expressing the indicator function for $n$ in the form \eqref{incsume}, it is easy to find $c(n)$.
  
  The following is plain.
  
  \begin{theorem}
    Let
  \begin{equation}
  f=\sum^q_{j=1}\lambda_j I_{c_j}
\end{equation}
be a function, where the $c_1<\cdots<c_q$ are positive integers (possibly $q=0$) and $\lambda_1,\ldots,\lambda_q\ne 0$ are any integers. If $q>0$, then the smallest positive integer $k$ such that $f(k)\ne0$ is $c_1$. If $q=0$, then for all integers $k\ge1$, $f(k)=0$.
  \end{theorem}

As a consequence of the above theorem, we have the following corollary.

\begin{corollary}\label{findc}
  Suppose that the indicator function for $n$ is expressed as
 \begin{equation}
   I=\sum^q_{j=1}\lambda_j I_{c_j},
 \end{equation}
 where $q\ge0$, $0<c_1<\ldots<c_q$, and $\lambda_1,\ldots,\lambda_q\ne0$ are integers. Then $c(n)=c_1$ when $q>0$, and $c(n)=\infty$ when $q=0$.
\end{corollary}

In this way, once we have expressed the indicator function of a number $n$ into the form \eqref{incsume}, it will be straightforward to determine both $\omega_0(n)$ and $c(n)$, using Corollaries \ref{findw0} and \ref{findc}, respectively. In the next section, we describe the general procedure, starting from a given $n$ as in Theorem \ref{main1}, to eventually express its indicator function into the form \eqref{incsume}.
  
\section{General procedure}\label{procedure} Throughout this section, we fix a natural number $n$ as in Theorem \ref{main1}, i.e.\ $n$ is not a multiple of $10$, and not a palindrome. The following describes a general procedure, consisting of a few steps, to express the indicator function $I$ for $n$ into the form \eqref{incsume}, which can be used to determine both $\omega_0(n)$ and $c(n)$. This procedure works due to the previous discussions.

\subsection{Step 1.} Factorize both $n$ and $r(n)$,
\begin{align}
  n & =p^{a_1}_1\cdots p^{a_m}_m, \\
  r(n) & = p^{b_1}_1\cdots p^{b_m}_m,
\end{align}
where $p_1<\cdots<p_m$ are primes, and $a_i,b_i\ge0$ are integers, not both $0$.

\subsection{Step 2.} Look for those primes $p_i$ for which $a_i\ne b_i$, i.e.\ the crucial primes. Since we are only going to focus on these primes, we denote them again by $p_1<\cdots <p_m$, and the exponents are $a_i,b_i$. Define the numbers $\delta_i=a_i-b_i$, $\mu_i=\min(a_i,b_i)$, for $1\le i\le m$.

\subsection{Step 3.} The characteristic equation for $n$ is
\begin{equation}
  \sgn(\delta_1)u_1+\sgn(\delta_2)u_2+\cdots+\sgn(\delta_m)u_m=0.
\end{equation}
We want to solve it for $u_i\in R_{p_i,|\delta_i|}$, i.e.\ to find the characteristic solutions. If there are no solutions, then conclude that $c(n)=\infty$ and $\omega_0(n)=1$. Otherwise, let the solutions be $\mathbf{u}_1,\ldots,\mathbf{u}_t$, in any order.

\subsection{Step 4.} For each characteristic solution $\mathbf{u}$, we have the sets $A_{\mathbf{u}}$ and $B_{\mathbf{u}}$ of Definition \ref{AuBu}. The solution $\mathbf{u}$ is nondegenerate if and only if $S(A_\mathbf{u},B_\mathbf{u})\ne\varnothing$. Now $S(A_\mathbf{u},B_\mathbf{u})\ne\varnothing$ if and only if $b\nmid \lcm(A_\mathbf{u})$ for all $b\in B_\mathbf{u}$. Use this to rule out those characteristic solutions $\mathbf{u}$ which are degenerate. If no characteristic solutions remain, conclude that $c(n)=\infty$ and $\omega_0(n)=1$. Otherwise, let the nondegenerate characteristic solutions be $\mathbf{u}^\ast_1,\ldots,\mathbf{u}^\ast_s$, in any order.

\subsection{Step 5.} The indicator function $I^n$ for $n$ is then given by Theorem \ref{indicator} as
\begin{equation}
  I^n=\sum^s_{i=1}I_{S_{\mathbf{u}^\ast_i}}.
\end{equation}
By Corollary \ref{indicatoru} this can be written as
\begin{equation}
  I^n=\sum^s_{i=1}I_{\operatorname{lcm}(A_{\mathbf{u}^\ast_i})}\prod_{b\in B_{\mathbf{u}^\ast_i}}(1-I_b).
\end{equation}
 Multiplying everything out on the right-hand side above with the help of Lemma \ref{comb} and collecting like terms, $I^n$ can be expressed into the form \eqref{incsume}, i.e.\ 
 \begin{equation}
   I^n=\sum^q_{j=1}\lambda_j I_{c_j},
 \end{equation}
 where $q\ge1$, $0<c_1<\cdots<c_q$, and $\lambda_1,\ldots,\lambda_q\ne0$ are integers (how this is actually done is illustrated in the example of $n=126$ in Section \ref{counterexample}). Finally, conclude that
 \begin{align}
   c(n)=c_1,\quad \omega_0(n)=\lcm(c_1,\ldots,c_q).
 \end{align}

\subsection{Some remarks about the procedure.}
We have described the general procedure in five steps as above. Whether $c(n)=\infty$ can be ascertained at certain points during the procedure. Namely, in Step 3, if there are no characteristic solutions at all, we immediately conclude that $c(n)=\infty$ and the procedure ends; and in Step 5, if all the characteristic solutions are degenerate, then we immediately conclude that $c(n)=\infty$ and the procedure ends. Otherwise, $c(n)<\infty$, $\omega_0(n)$, and the indicator function $I^n$ are found in Step 5.

\section{Counterexample to Conjecture \ref{conj1}.}\label{counterexample}

The smallest counterexample to Conjecture \ref{conj1} is found by PARI/GP \cite{T} to be $n=126$. We perform the general procedure of Section \ref{procedure} to $n=126$ as follows.

\subsection{Step 1.} We factorize
\begin{align}
  126 &=  2\cdot 3^2\cdot 7,\\
  621 & = 3^3\cdot 23.
\end{align}

\subsection{Step 2.} The crucial primes are $2,3,7,23$. We arrange the numbers $p_i$, $a_i$, $b_i$, $\delta_i$ and $\mu_i$ into a table.

\begin{table}[H]
 \caption{$p_i$, $a_i$, $b_i$,$\delta_i$, and $\mu_i$ for $n=126$.}
 \label{table:deltamu}
 \centering
  \begin{tabular}{llllll}
   \hline
   $i$ & $p_i$ & $a_i$ & $b_i$ & $\delta_i$  & $\mu_i$ \\
   \hline \hline
   $1$ & $2$ & $1$ &$0$ &$1$  &$0$\\
   $2$ & $3$ & $2$ &$3$ &$-1$ &$2$\\
   $3$ & $7$ & $1$& $0$ &$1$ &$0$\\
   $4$ & $23$ & $0$ &$1$& $-1$ &$0$\\
   
   \hline
  \end{tabular}
\end{table}

\subsection{Step 3.} The characteristic equation is
\begin{equation}
  u_1-u_2+u_3-u_4=0,
\end{equation}
where we want to solve for $u_1\in\{1,2\}$, $u_2\in\{1,2,3\}$, $u_3\in\{1,2,7\}$, and $u_4\in \{1,2,23\}$. The characteristic solutions are
\begin{gather}
  \mathbf{u}_1 =(1,1,1,1),\quad \mathbf{u}_2  = (1,1,2,2), \quad \mathbf{u}_3  = (1,2,2,1),\quad \mathbf{u}_4  = (2,1,1,2)  \\
  \mathbf{u}_5  = (2,2,1,1),\quad 
  \mathbf{u}_6  = (2,2,2,2),\quad 
  \mathbf{u}_7  = (2,3,2,1).
\end{gather}
For each characteristic solution $\mathbf{u}_l$ ($1\le l\le 7$), also write $\mathbf{u}_l=(u_{l1},u_{l2},u_{l3},u_{l4})$.

\subsection{Step 4.} We make two tables of the crucial primes $p_i$ ($1\le i\le 4$) versus the characteristic solutions $\mathbf{u}_l$ ($1\le l\le 7$) as follows.

The first is where in the $(p_i,\mathbf{u}_l)$-entry we have the $D(p_i,|\delta_i|,u_{li},\mu_i)$ of Definition \ref{defnD}. The second is where in the $(p_i,\mathbf{u}_l)$-entry we have the $T_{p_i,\mathbf{u}_l}$ defined in \eqref{defnT1}, \eqref{defnT2}, and \eqref{defnT3}, and also at the bottom, the sets $A_{\mathbf{u}}$, $B_{\mathbf{u}}$, and $S_{\mathbf{u}}$. The first table helps us construct the second table because the definition of $T_{p_i,\mathbf{u}_l}$ depends on $D(p_i,|\delta_i|,u_{li},\mu_i)$.

\begin{table}[H]
 \caption{Table of $D(p_i,|\delta_i|,u_{li},\mu_i)$.}
 \label{table:indicatorfunB}
 \centering
  \begin{tabular}{llllllll}
   \hline
   $$ & $\mathbf{u}_1$ & $\mathbf{u}_2$ & $\mathbf{u}_3$& $\mathbf{u}_4$& $\mathbf{u}_5$& $\mathbf{u}_6$& $\mathbf{u}_7$\\
   \hline \hline
   $2$ & $[\mathrm{v}]$ & $[\mathrm{v}]$ & $[\mathrm{v}]$& $[\mathrm{iii}]$& $[\mathrm{iii}]$& $[\mathrm{iii}]$& $[\mathrm{iii}]$\\
   $3$ & $[\mathrm{vi}]$ & $[\mathrm{vi}]$ & $[\mathrm{vii}]$& $[\mathrm{vi}]$& $[\mathrm{vii}]$& $[\mathrm{vii}]$& $[\mathrm{vii}]$\\
   $7$ & $[\mathrm{v}]$ & $[\mathrm{ii}]$ & $[\mathrm{ii}]$& $[\mathrm{v}]$& $[\mathrm{v}]$& $[\mathrm{ii}]$& $[\mathrm{ii}]$\\
   $23$ & $[\mathrm{v}]$ & $[\mathrm{ii}]$ & $[\mathrm{v}]$& $[\mathrm{ii}]$& $[\mathrm{v}]$& $[\mathrm{ii}]$& $[\mathrm{v}]$\\
   \hline
  \end{tabular}
\end{table}

\begin{table}[H]
 \caption{Table of $T_{p_i,\mathbf{u}_l}$ and $A_\mathbf{u}$, $B_\mathbf{u}$, and $S_{\mathbf{u}}$.}
 \label{table:indicatorfunC}
 \centering
  \begin{tabular}{llllllll}
   \hline
   $$ & $\mathbf{u}_1$ & $\mathbf{u}_2$ & $\mathbf{u}_3$& $\mathbf{u}_4$& $\mathbf{u}_5$& $\mathbf{u}_6$& $\mathbf{u}_7$\\
   \hline \hline
   $2$ & $(\varnothing,\{1\})$ & $(\varnothing,\{1\})$ & $(\varnothing,\{1\})$& $(\varnothing,\varnothing)$& $(\varnothing,\varnothing)$& $(\varnothing,\varnothing)$& $(\varnothing,\varnothing)$\\
   $3$ & $(\varnothing,\varnothing)$ & $(\varnothing,\varnothing)$ & $(\varnothing,\{1\})$& $(\varnothing,\varnothing)$& $(\varnothing,\{1\})$& $(\varnothing,\{1\})$& $(\varnothing,\{1\})$\\
   $7$ & $(\{14\},\varnothing)$ & $(\{2\},\{14\})$ & $(\{2\},\{14\})$& $(\{14\},\varnothing)$& $(\{14\},\varnothing)$& $(\{2\},\{14\})$& $(\{2\},\{14\})$\\
   $23$ & $(\{506\},\varnothing)$ & $(\{22\},\{506\})$ & $(\{506\},\varnothing)$& $(\{22\},\{506\})$& $(\{506\},\varnothing)$& $(\{22\},\{506\})$& $(\{506\},\varnothing)$\\
   \hline
   $A_{\mathbf{u}}$ & $\{14,506\}$ & $\{2,22\}$ & $\{2,506\}$& $\{14,22\}$& $\{14,506\}$& $\{2,22\}$& $\{2,506\}$\\
   
   $B_{\mathbf{u}}$ & $\{1\}$ & $\{1,14,506\}$ & $\{1,14\}$& $\{506\}$& $\{1\}$& $\{1,14,506\}$& $\{1,14\}$\\
  
   $S_\mathbf{u}$ & $\varnothing$ & $\varnothing$ & $\varnothing$& $S(\{14,22\},\{506\})$& $\varnothing$& $\varnothing$& $\varnothing$\\
  \end{tabular}
\end{table}

We see immediately from the above table that the only nondegenerate solution is $\mathbf{u}_4$.

\subsection{Step 5.} The indicator function for $126$ is then
\begin{equation}\label{ind126}
  I = I_{14} I_{22} (1-I_{506})= I_{154}- I_{3542}.
\end{equation}
We conclude that $c(126)=154$ and $\omega_0(126)=\lcm(154,3542)=3542$. Since $\omega_f(126) = 31878$ (calculation omitted), we see that $n=126$ is a counterexample to Conjecture \ref{conj1}.

\section{A more refined conjecture but still a counterexample.}
So it was not difficult to find a counterexample to Conjecture \ref{conj1}, because there is a counterexample as small as $126$. One can try to improve Conjecture \ref{conj1}. One possibility is to look at Theorem \ref{indicator}, where the quantity \eqref{p}, namely
\begin{equation}
  \omega_b(n)=\lcm\left(\bigcup_{\mathbf{u}\in U^\ast}(A_{\mathbf{u}}\cup B_{\mathbf{u}})\right)
\end{equation}
is a better candidate for a small period. Thus it would be natural, to mimic Conjecture \ref{conj1}, to conjecture that $\omega_0(n)$ is always $1$ or $\omega_b(n)$. But this too, is false, as the smallest counterexample found by PARI/GP \cite{T} is $n=5957$.

\section{Table of indicator functions.}

We provide a table of the indicator functions, thereby fundamental periods and orders, for some numbers $n$. These are calculated by using PARI/GP \cite{T}.

\begin{table}[H]
 \caption{Indicator functions for some numbers $n$.}
 \label{table:indicatorfunA}
 \centering
  \begin{tabular}{llll}
   \hline
   $n$ & $I^n$ & $c(n)$ & $\omega_0(n)$ \\
   \hline \hline
   $13$ & $I_{15}-I_{195}-I_{465}+2I_{6045}$ & $15$ & $6045$ \\
   $17$ & $I_{280}-I_{4760}-I_{19880}+2I_{337960}$ & $280$ &$337960$ \\
   $18$ & $I_1$ & $1$ &$1$ \\
   $19$ & $I_{819}-I_{15561}$ & $819$ &$15561$ \\
   $26$ & $I_{15}-I_{195}-I_{465}+2I_{6045}$ & $15$ &$6045$ \\
   $37$ & $I_{12}- I_{444}- I_{876} +2I_{32412}$ &  $12$ &$32412$\\
   $39$ & $I_{15}-I_{195}-I_{465}+2I_{6045}$ & $15$ &$6045$ \\
   $48$ & $I_{3}-I_{21}$ &$3$ & $21$ \\
   $49$ & $I_{3243}-I_{22701}$ & $3243$ &$22701$ \\
   $56$ & $I_{3}-I_{21}-I_{39}+2I_{273}$ & $3$ &$273$ \\

$79$ & $I_{624}-I_{ 49296}-I_{ 60528}+2I_{ 4781712}$ & $624$ &$4781712$\\
$103$ & $I_{10234}-I_{ 1054102}$ & $10234$ & $1054102$\\

$107$ & $I_{37100}-I_{3969700}-I_{26007100}+2I_{2782759700}$ &$37100$ & $2782759700$\\
 $109$ & $I_{1686672}-I_{183847248}$ &$1686672$ & $183847248$\\
 $113$ & $I_{17360}-I_{1961680}-I_{ 5398960}+2I_ {610082480}$ & $17360$ &$610082480$\\
 $117$ & $I_{2054}$ & $2054$ &$2045$\\
 $119$ & $I_{123760}-I_{ 112745360}$ & $123760$ &$112745360$\\
 $122$ & $I_{80}-I_{1040}-I_{1360}-I_{4880}+I_{17680}+2I_{63440}+2I_{82960}-3I_{1078480}$ & $80$ &$1078480$ \\
   \hline
  \end{tabular}
\end{table}

We see that the indicator functions for $13$, $26$, and $39$ are identical. There is another curiosity in the above table, we see that for all these indicator functions, the largest subscript is a multiple of all smaller subscripts. This is not always true, and the smallest counterexample, found by PARI/GP \cite{T} is $n=21726$, with
\begin{gather*}
  I^{21726}=
I_{816}-I_{ 5712}-I_{ 8976}-I_{ 10608}+I_{ 16401}-I_{ 32802}+I_{ 62832}+I_{ 74256} \\
+I_{ 116688}-I_{ 816816}-I_{ 1098867}+I_{ 2197734}.
\end{gather*}
Here $816\nmid 2197734$.

\section{Appendix on periodic arithmetical functions.}

We first recall some basic properties of periodic arithmetical functions in Section \ref{5.1}. There, we prove a formula for the fundamental period of an arbitrary periodic arithmetical function $\mathbb{Z}\to\mathbb{C}$ (Theorem \ref{funperiod}). This formula is actually equivalent to the formula given in Theorem 9 in \cite{V1}, and we prove their equivalence in Section \ref{5.2}. Therefore Theorem \ref{funperiod} is not a new result. In Section \ref{12.3}, we prove Theorem 12 in \cite{V1} (labeled as Theorem \ref{Theorem12} in this paper) using Theorem \ref{funperiod}. In this way, our paper becomes more self-contained, with the logical dependencies as follows.

\begin{itemize}
  \item Though equivalent to Theorem 9 in \cite{V1}, Theorem \ref{funperiod} is proved from first principles in this paper. 
  \item Though the same as Theorem 12 in \cite{V1}, Theorem \ref{Theorem12} is proved by using Theorem \ref{funperiod}.
  \item Theorem \ref{Thm12} is proved by using Theorem \ref{Theorem12}. As a consequence, we have Corollary \ref{findw0}.
\end{itemize}

The paper \cite{V1} contains both its Theorems 9 and 12, with proofs. However, its proof of Theorem 12 does not seem to be a direct application of its Theorem 9.

\subsection{Basic properties and a formula for the fundamental period.}\label{5.1}

Let the function $e\colon \mathbb{R}\to \mathbb{C}$ be defined by
\begin{equation}
  e(t)=e^{2\pi i t}.
\end{equation}
Then the set of all roots of unity in $\mathbb{C}$ is
\begin{equation}
  R=\{e(\alpha)\mid \alpha\in\mathbb{Q}\}.
\end{equation}
For a $\zeta=e(\alpha)\in R$, where $0\le\alpha < 1$ is rational, write $\alpha=a/b$ in lowest terms, i.e.\ $a,b\in\mathbb{Z}$, $b>0$, and $(a,b)=1$, then denote $\nu(\zeta)=a$ and $\delta(\zeta)=b$. Thus $\zeta$ is a primitive $\delta(\zeta)$-th root of unity. For each integer $m\ge1$, we denote by $\zeta_m$ the primitive $m$-th root of unity $e(1/m)$.

Consider functions $g\colon R\to \mathbb{C}$ with $g(\zeta)=0$ outside a finite set. Let the set of all such functions be denoted $\mathcal{G}$. For a $g\in\mathcal{G}$, define
an arithmetical function $f\colon\mathbb{Z}\to\mathbb{C}$ by \begin{equation}\label{f}
  f(x)=\sum_{\zeta\in R}g(\zeta)\zeta^x.
\end{equation}
The sum is actually finite because $g(\zeta)=0$ for all but finitely many $\zeta$'s. We denote this $f$ by $\Phi(g)$. We prove in Theorem \ref{period} that $f$ is a periodic function. Let us recall some definitions.

\begin{definition}
  A function $f\colon\mathbb{Z}\to\mathbb{C}$ (respectively $f\colon\mathbb{N}\to\mathbb{C}$) is {\it periodic} if there is an integer $\omega>0$ such that for all $x\in\mathbb{Z}$ (respectively $x\in\mathbb{N}$),
  \begin{equation}
    f(x+\omega)=f(x).
  \end{equation}
  Such an $\omega$ is called a {\it period} of $f$, and we also say that $f$ is {\it periodic modulo} $\omega$. When $f$ is periodic, the smallest period of $f$ is called its {\it fundamental period}.
\end{definition}

We have the following characterization of periods.

\begin{theorem}\label{restriction}
Let $f\colon \mathbb{Z}\to\mathbb{C}$ be a periodic function. Then we have the following.
\renewcommand{\labelenumi}{(\roman{enumi})}
  \begin{enumerate}
    \item $f|_\mathbb{N}$, the restriction of $f$ to $\mathbb{N}$, is periodic. Moreover, an integer $\omega>0$ is a period of $f$ if and only if it is a period of $f|_\mathbb{N}$.
    \item The periods of $f$ are precisely the positive integral multiples of its fundamental period $\omega_0$. \label{ii}
  \end{enumerate}
\end{theorem}

\begin{proof}
\renewcommand{\labelenumi}{(\roman{enumi})}
  \begin{enumerate}
    \item Plainly any period of $f$ is also a period of its restriction $f|_\mathbb{N}$. We only need to prove that any period of $f|_\mathbb{N}$ is conversely a period of $f$. So let $\omega$ be a period of $f|_\mathbb{N}$. Choose a period $\mu$ of $f$, then $\mu$ is also a period of $f|_\mathbb{N}$. For any $x\in\mathbb{Z}$, there exists a positive integer $q>0$ such that $x+q\mu>0$. Also, $x+\omega+q\mu>0$. Since $\omega$ is a period of $f|_\mathbb{N}$, $f(x+q\mu)=f(x+q\mu+\omega)$. Now, since $\mu$ is a period of $f$,
    \begin{equation}
      f(x)=f(x+q\mu)=f(x+q\mu+\omega)=f(x+\omega).
    \end{equation}
    Since the above holds for all $x\in\mathbb{Z}$, $\omega$ is a period of $f$.
    \item Let $\omega$ be a period of $f$. Use the division algorithm to write $\omega=q\omega_0+r$, where $q,r\in\mathbb{Z}$ are such that $0\le r<\omega_0$ and $q>0$. Assume that $r>0$, then $r=\omega-q\omega_0$. For any $x\in\mathbb{Z}$,
    \begin{equation}
      f(x)=f(x+\omega)=f(x+\omega-q\omega_0)=f(x+r),
    \end{equation}
    because both $\omega$ and $\omega_0$ are periods of $f$. Hence $r$ is a period of $f$ smaller than $\omega_0$, this is a contradiction. Hence $r=0$ and so $\omega_0\mid \omega$. The converse, that any positive integral multiple of $\omega_0$ is a period of $f$, is plain.
  \end{enumerate}
  
\end{proof}

For the rest of this section, we shall deal only with arithmetical functions $f\colon\mathbb{Z}\to\mathbb{C}$ defined for every integer.

\begin{theorem}\label{period}
  Let $g\in\mathcal{G}$, then $\Phi(g)$ is periodic modulo
  \begin{equation}\label{lcm}
    \lcm\{\delta(\zeta)\mid \zeta\in R,g(\zeta)\ne0\}.
  \end{equation}
\end{theorem}
\begin{proof}
  Let the least common multiple \eqref{lcm} be denoted by $\omega$ and $\Phi(g)=f$. For each $\zeta\in R$ with $g(\zeta)\ne0$, $\zeta$ is a $\delta(\zeta)$-th root of unity, therefore $\zeta^{x+\delta(\zeta)}=\zeta^x$ for every $x\in\mathbb{Z}$. As $\omega$ is a multiple of $\delta(\zeta)$, $\zeta^{x+\omega}=\zeta^x$ for every $x\in\mathbb{Z}$. Consequently, for every $x\in\mathbb{Z}$,
  \begin{equation}
    f(x+\omega)=\sum_{\zeta\in R,g(\zeta)\ne0}g(\zeta)\zeta^{x+\omega}=\sum_{\zeta\in R,g(\zeta)\ne0}g(\zeta)\zeta^{x}=f(x).
  \end{equation}
\end{proof}

Let the set of all periodic arithmetical functions $f\colon\mathbb{Z}\to\mathbb{C}$ be denoted by $\mathcal{F}$. The above established a mapping $\Phi\colon\mathcal{G}\to\mathcal{F}$. We shall prove that it is bijective.

\begin{theorem}
  The mapping $\Phi\colon\mathcal{G}\to\mathcal{F}$ is bijective.
\end{theorem}
\begin{proof}
  Let $f\in\mathcal{F}$ be periodic modulo $\omega$. By Theorem 8.4 on p.\ 160 in \cite{A}, there exist unique coefficients $h_r\in\mathbb{C}$ for $0\le r<\omega$ such that for all $x\in\mathbb{Z}$,
  \begin{equation}
    f(x)=\sum^{\omega-1}_{r=0}h_r\zeta^{xr}_\omega.
  \end{equation}
   If we define the function $g\colon R\to\mathbb{C}$ by setting $g(\zeta^r_\omega)=h_r$ for $0\le r<\omega$ and $g(\zeta)=0$ for all other $\zeta\in R$, it is easy to see that $g\in\mathcal{G}$ and that $\Phi(g)=f$. Whence $\Phi$ is surjective.
  
  We now prove injectivity. Assume that $\Phi(g_1)=\Phi(g_2)=f\in\mathcal{F}$, then
for all $x\in\mathbb{Z}$,  \begin{equation}
    f(x)=\sum_{\zeta\in R}g_1(\zeta)\zeta^x=\sum_{\zeta\in R}g_2(\zeta)\zeta^x.
  \end{equation}
  Let $S=\{\zeta\in R\mid(g_1(\zeta),g_2(\zeta))\ne(0,0)\}$, then $S$ is finite and the above sums can be written as
  \begin{equation}
    \sum_{\zeta\in S}g_1(\zeta)\zeta^x=\sum_{\zeta\in S}g_2(\zeta)\zeta^x.
  \end{equation}
  Consequently, for all $x\in\mathbb{Z}$,
  \begin{equation}\label{difference}
    \sum_{\zeta\in S}(g_1(\zeta)-g_2(\zeta))\zeta^x=0.
  \end{equation}
  If $S=\varnothing$, plainly $g_1=g_2=0$ is identically zero. Thus assume otherwise and list the elements of $S$ as $\{\xi_1,\ldots,\xi_m\}$. Put $x_j=g_1(\xi_j)-g_2(\xi_j)$ for $1\le j\le m$. Then \eqref{difference} becomes
  \begin{equation}
    \sum^m_{j=1}x_j\xi^x_j=0.
  \end{equation}
  Since this holds for all $x\in\mathbb{Z}$, in particular it holds for all $0\le x<m$, and we have a homogeneous system of linear equations. Since the Vandermonde determinant
  \begin{equation}
  \begin{vmatrix}
    1 & 1 & \cdots & 1 \\
    \xi_1 & \xi_2 & \cdots & \xi_m \\
    \cdots & \cdots & \cdots & \cdots \\
    \xi^{m-1}_1 & \xi^{m-1}_2 & \cdots & \xi^{m-1}_m
  \end{vmatrix}\ne0
  \end{equation}
  as the $\xi_j$'s are distinct,  $x_j=0$ for all $1\le j\le m$, i.e.\ $g_1(\xi_j)=g_2(\xi_j)$ for all $1\le j\le m$. In other words, $g_1(\zeta)=g_2(\zeta)$ for all $\zeta\in S$. Since $g_1(\zeta)=g_2(\zeta)=0$ for all $\zeta\in R\setminus S$, we have shown that $g_1=g_2$. Whence $\Phi$ is injective.
\end{proof}
  
\begin{theorem}\label{funperiod}
  Let $g\in\mathcal{G}$, then the fundamental period of $\Phi(g)$ is indeed given by \eqref{lcm}.
\end{theorem}
\begin{proof}
  Let $\Phi(g)=f$. Theorem \ref{period} already showed that
  \begin{equation}\label{k}
    \omega=\lcm\{\delta(\zeta)\mid \zeta\in R,g(\zeta)\ne0\}
\end{equation}
is a period of $f$. We need to show that it is the smallest one. Assume that the smallest period is actually $\omega_0$, where $0<\omega_0<\omega$. By part (ii) of Theorem \ref{restriction}, $\omega_0\mid \omega$. By Theorem 8.4 on p.\ 160 in \cite{A}, there exist unique coefficients $h_r$ for $0\le r<\omega_0$ such that for all $x\in\mathbb{Z}$,
\begin{equation}
    f(x)=\sum^{\omega_0-1}_{r=0}h_r\zeta^{xr}_{\omega_0}.
  \end{equation}
  Hence we see that $g(\zeta^r_{\omega_0})=h_r$ for $0\le r<\omega_0$ and $g(\zeta)=0$ for all other $\zeta\in R$.
  
  Now $\omega_0<\omega$ and so in view of \eqref{k} there exists some $\xi\in R$ with $g(\xi)\ne0$ such that $\delta(\xi)\nmid \omega_0$. We have
  \begin{equation}
   \xi^{\omega_0}=\left(e\left(\frac{\nu(\xi)}{\delta(\xi)}\right)\right)^{\omega_0}=e\left(\frac{\nu(\xi)\omega_0}{\delta(\xi)}\right).
  \end{equation}
  Now the argument on the right above is not an integer. For if it is, then $\delta(\xi)\mid \nu(\xi)\omega_0$. Since $(\delta(\xi),\nu(\xi))=1$, $\delta(\xi)\mid \omega_0$, which is a contradiction. Therefore $\nu(\xi)\omega_0/\delta(\xi)$ is not an integer, and so $\xi^{\omega_0}\ne1$. That is, $\xi$ is not a $\omega_0$-th root of unity. But $g$ vanishes at all $\zeta\in R$ which is not a $\omega_0$-th root of unity. This is a contradiction. Hence $\omega$ is indeed the fundamental period of $f$.
\end{proof}

\subsection{Equivalence of Theorem \ref{funperiod} and Theorem 9 in \cite{V1}.}\label{5.2}

Let $f\colon\mathbb{Z}\to\mathbb{C}$ be a periodic arithmetical function of period $\omega$, then by Theorem 8.4 on p.\ 160 in \cite{A}, there exist unique coefficients $h_r$ for $0\le r<\omega$ such that for all $x\in\mathbb{Z}$,
\begin{equation}
  f(x)=\sum^{\omega-1}_{r=0}h_r\zeta^{xr}_{\omega}.
\end{equation} 
We can write $f(x)$ as
\begin{equation}
  f(x)=\sum^\omega_{k=1}h_{\omega-k}\zeta^{-xk}_\omega,
\end{equation}
or if we rename $h_{\omega-k}$ as $h_k$,
\begin{equation}
  f(x)=\sum^\omega_{k=1}h_{k}\zeta^{-xk}_\omega.
\end{equation}
Let the set of $1\le k\le \omega$ such that $h_k\ne0$ be $\{k_1,\ldots,k_l\}$, then according to Theorem 9 in \cite{V1}, the fundamental period of $f$ is
\begin{equation}\label{hisp}
  \omega_0=\frac{\omega}{(k_1,\ldots,k_l,\omega)},
\end{equation}
where the parentheses denote a greatest common divisor.
On the other hand, according to Theorem \ref{funperiod},
\begin{equation}\label{myp}
  \omega_0=\lcm(\delta(\zeta^{-k_1}_{\omega}),\ldots,\delta(\zeta^{-k_l}_{\omega})).
\end{equation}
We show that \eqref{hisp} is equal to \eqref{myp}, i.e.\
\begin{equation}\label{need}
  \frac{\omega}{(k_1,\ldots,k_l,\omega)}=\lcm(\delta(\zeta^{-k_1}_{\omega}),\ldots,\delta(\zeta^{-k_l}_{\omega})).
\end{equation}
When there are no $1\le k\le \omega$ such that $h_k\ne0$, i.e.\ when $l=0$, this is plain. Thus assume that $l>0$. Notice that for $1\le j\le l$,
\begin{equation}\label{act}
  \delta(\zeta^{-k_j}_\omega)=\delta\left(e\left(\frac{-k_j}{\omega}\right)\right)=\frac{\omega}{(k_j,\omega)},
\end{equation}
which is easily seen to divide the left-hand-side of \eqref{need}. Hence in \eqref{need}, the left-hand side is a multiple of the right-hand side.

Conversely, denote the right-hand side of \eqref{need} by $M$. Notice that for $1\le j\le l$, because of \eqref{act},
\begin{equation}\label{since}
  \frac{\omega}{(k_j,\omega)} \mid M,\quad\text{and therefore}\quad \omega\mid M (k_j,\omega).
\end{equation}
Since
\begin{equation}\label{comp}
  (k_1,\ldots,k_l,\omega)=((k_1,\omega),\ldots,(k_l,\omega)),
\end{equation}
we have a linear combination
\begin{equation}\label{comb2}
  (k_1,\ldots,k_l,\omega)=\sum^l_{j=1}y_j(k_j,\omega),
\end{equation}
where the $y_j$'s are integers. We prove that the right-hand side $M$ of \eqref{need} is a multiple of the left-hand side, or equivalently,
\begin{equation}
  \omega \mid M (k_1,\ldots,k_l,\omega).
\end{equation}
Since \eqref{since} holds for all $1\le j\le l$, using also \eqref{comb2},
\begin{equation}
  \omega\mid \sum^l_{j=1}y_jM (k_j,\omega)=M\sum^l_{j=1}y_j (k_j,\omega)=M(k_1,\ldots,k_l,\omega).
\end{equation}
Since we have proved that each side of \eqref{need} is a multiple of the other side, equality holds.

In summary, Theorem \ref{funperiod} and Theorem 9 in \cite{V1} both give a formula for the fundamental period of a periodic arithmetical function. These formulae might look different on the surface, but indeed give the same fundamental period.

\subsection{Proof of Theorem \ref{Theorem12} using Theorem \ref{funperiod}}\label{12.3}

  We have borrowed Theorem 12 in \cite{V1} (Theorem \ref{Theorem12} in this paper) in our proof of Theorem \ref{Thm12}. To make this paper self-contained, we provide a proof of Theorem \ref{Theorem12} using Theorem \ref{funperiod}.

\begin{proof}
  For each $1\le\ j\le m$, we have
  \begin{equation}
    f_j(x)=\sum_{\zeta\in R^\ast(\omega_j)}g_j(\zeta)\zeta^x,
  \end{equation}
  where $g_j=\Phi^{-1}(f_j)$.
  Since the $R^\ast(\omega_j)$ are pairwise disjoint, $\Phi^{-1}(f)=g$, where
  \begin{equation}
    g(\zeta)=\begin{cases}
      g_j(\zeta) & \text{if $\zeta\in R^\ast(\omega_j)$ for some $1\le j\le m$,} \\
      0 & \text{otherwise.}
    \end{cases}
  \end{equation}
  By Theorem \ref{funperiod}, the fundamental period of $f$ is
  \begin{equation}
    L=\lcm\{\delta(\zeta)\mid g(\zeta)\ne0\},
  \end{equation}
  whereas the fundamental period asserted by the theorem is
  \begin{equation}
    T=\lcm\{\omega_1,\ldots,\omega_m\}.
  \end{equation}
  So we have to prove that $L=T$.
  
  Let $\zeta$ be a root of unity such that $g(\zeta)\ne0$, then $\zeta\in R^\ast(\omega_j)$ for some $1\le j \le m$. Since $\delta(\zeta)=\omega_j$ and $\omega_j\mid T$, we have $\delta(\zeta)\mid T$. Since $\delta(\zeta)\mid T$ for any root of unity $\zeta$ such that $g(\zeta)\ne0$, we have $L\mid T$. On the other hand, let $1\le j\le m$. Since $f_j\ne 0$, $g(\zeta)=g_j(\zeta)\ne0$ for some $\zeta\in R^\ast(\omega_j)$. Since $\delta(\zeta)=\omega_j$ and $\delta(\zeta)\mid L$, we have $\omega_j\mid L$. Since $\omega_j\mid L$ for any $1\le j\le m$, we have $T\mid L$. Since both $L\mid T$ and $T\mid L$, we have $L=T$.
\end{proof}

\section{Acknowledgment}
The author wish to thank Professor Kohji Matsumoto for the careful reading of this paper and Doctor Yuta Suzuki for kindly directing the author to the papers \cite{R}, \cite{V1}, \cite{V2}, \cite{V3}, and providing comments which greatly improved the presentation of this paper.

\end{document}